\documentclass[12pt]{amsart}
\usepackage{amsmath, amssymb, amsthm}
\usepackage{comment, empheq}
\usepackage{mathdots, breqn}
\usepackage{graphicx}
\usepackage{hyperref}
\usepackage{mathtools}
\newtheoremstyle{myremark}{10pt}{10pt}{}{}{\scshape}{.}{.5em}{}
\newtheorem{theorem}{Theorem}
\newtheorem{lemma}{Lemma}

\theoremstyle{remark}
\theoremstyle{myremark}
\newtheorem*{remark}{Remark}
\newtheorem*{ackno}{Acknowledgements}

\newcommand{\R}{\mathbb{R}}
\renewcommand{\H}{\mathbb{H}}
\newcommand{\C}{\mathbb{C}}

\newcommand{\Z}{\mathbb{Z}}

\newcommand{\N}{\mathbb{N}}

\newcommand{\T}{\mathbb{T}}
\newcommand{\B}{\mathcal{B}}
\newcommand{\D}{\mathcal{D}}

\newcommand{\GL}{\mathrm{GL}}

\newcommand{\PGL}{\mathrm{PGL}}
\newcommand{\SL}{\mathrm{SL}}

\newcommand{\Lie}{\mathrm{Lie}}
\newcommand{\sgn}{\mathrm{sgn}} 
\newcommand{\Eis}{\mathrm{Eis}}
\newcommand{\supp}{\mathrm{supp}}

\usepackage{cleveref}
\title[Joint equidistribution]{Joint equidistribution on the product of the circle and the unit tangent bundle of the modular surface}
\author{Subhajit Jana}

\address{Max Planck Institute for Mathematics, Vivatgasse 7, 53111 Bonn, Germany.}
\email{subhajit@mpim-bonn.mpg.de}

\begin{document}

\begin{abstract}
    We use spectral method to prove a joint equidistribution of primitive rational points and the same along expanding horocycle orbits in the products of the circle and the unit cotangent bundle of the modular surface. This result explicates the error bound in a recent work of Einsiedler, Luethi, and Shah \cite[Theorem $1.1$]{ELS}. The error is sharp upon the best known progress towards the Ramanujan conjecture at the finite places for the modular surface. 
\end{abstract}

\maketitle

\section{Introduction}
Let $n\in \N$ be tending off to infinity. Let $G:=\PGL_2(\R)$, and $\Gamma:=\SL_2(\Z)$, and finally, let $\T$ be the one torus $\R/\Z$. We consider the set of points in $\T\times\Gamma\backslash G$
$$\mathcal{P}(n):=\left\lbrace\left(\frac{k}{n}, \begin{pmatrix}\sqrt{n^{-1}}&k\sqrt{n^{-1}}\\&\sqrt{n}\end{pmatrix}\right)\in \T\times\Gamma\backslash G: (k,n)=1 \right\rbrace.$$
The points $\begin{pmatrix}\sqrt{n^{-1}}&k\sqrt{n^{-1}}\\&\sqrt{n}\end{pmatrix}$ lie on the expanding horocycles in the unit cotangent bundle of the modular surface $\Gamma\backslash\H$ where $\H$ is the upper half plane.
A recent result of Einsiedler, Luethi, and Shah \cite[Theorem $1.1$]{ELS} shows an effective equidistribution of $\mathcal{P}(n)$ on $\T\times\Gamma\backslash G$. Their result states that there exist $\kappa>0$ such that for any test function $f\in C_c^\infty(\T\times\Gamma\backslash G)$ one has
\begin{equation}\label{els}
    \frac{1}{|\mathcal{P}(n)|}\sum_{k\in\mathcal{P}(n)}f(k)-\int_{\T\times\Gamma\backslash G} f(x) dx \ll_{f} n^{-\kappa},
\end{equation}
where $dx$ is the product of the probability Haar measures on $\T$ and $\Gamma\backslash G$ and the implied constant depends on $f$ through some Sobolev norm of $f$. However, the ergodic method used in \cite{ELS} to prove \eqref{els} is not sufficient to give a sharp lower bound of $\kappa$.

In this paper, using spectral theory of $\T\times\Gamma\backslash G$, we produce an explicit error bound, i.e. a lower bound of $\kappa$, in \eqref{els}. This bound is essentially the sharpest possible which one can obtain using the spectral methods and progress towards the Ramanujan conjecture at the finite places for $\Gamma\backslash G$. Below we describe our main theorem.

\begin{theorem}\label{main-theorem}
Let $f\in C_c^\infty(\T\times\Gamma\backslash G)$. Let $n$ be tending off to infinity on the natural numbers. Then for any $\delta>7/64$,
\begin{equation*}
    \frac{1}{\varphi(n)}\sum_{(k,n)=1}f\left[\frac{k}{n}, \begin{pmatrix}\sqrt{n^{-1}}&k\sqrt{n^{-1}}\\&\sqrt{n}\end{pmatrix}\right]=\int_{\T \times\Gamma\backslash G}f+O_f(n^{-1/2+\delta}),
\end{equation*}
where $\varphi$ is the Euler totient function. The implicit constant in the error term depends on $f$ through some fixed degree Sobolev norm of $f$ (see \S\ref{sec:sobolev-norm} for the definitions) and the support of the projection $f$ on the test functions on $\Gamma\backslash G$.\footnote{One can modify the definition of the Sobolev norm in \S\ref{sec:sobolev-norm} by ``height to the cusp'', as in \cite{MV}, so that the dependency on the support condition can be absorbed into the Sobolev norm of the test function.}
\end{theorem}

\begin{remark}
We provide a sketch for the proof of Theorem \ref{main-theorem}.
Note that, the Schwartz kernel theorem implies that to prove Theorem \ref{main-theorem} it is enough to prove the same for $f:=f_1\otimes f_2\in C^\infty_c(\T)\otimes C^\infty_c(\Gamma\backslash G)$
\begin{equation}\label{effective-equidistribution}
    \frac{1}{\varphi(n)}\sum_{(k,n)=1}f_1\left(\frac{k}{n}\right)f_2\left[\begin{pmatrix}\sqrt{n^{-1}}&k\sqrt{n^{-1}}\\&\sqrt{n}\end{pmatrix}\right]-\int_\T f_1\int_{\Gamma\backslash G}f_2\ll_{\supp(f_2)} n^{-1/2+\delta}S_{d_1}(f_1)S_{d_2}(f_2),
\end{equation}
for some Sobolev norms $S_{d_i}$ with fixed $d_i\ge 0$. See \S\ref{sec:sobolev-norm} for the definitions of Sobolev norms.
We will decompose $f_1$ into its Fourier expansion in $L^2(\T)$ and $f_2$ into its spectral expansion in $L^2(\Gamma\backslash G)$. We will first prove \eqref{effective-equidistribution} for $f_1$ being a Fourier mode $e_l:z\mapsto e(lz):=e^{2\pi i l z}$ for some integer $l$ and $f_2$ being a Hecke eigenform $\phi$, cuspidal or Eisenstein series. Then we will use spectral decomposition of $L^2(\T)\otimes L^2(\Gamma\backslash G)$ to combine the above results to prove Theorem \ref{main-theorem}.
\end{remark}

\begin{remark}
In \cite[Theorem 1.1]{ELS} authors, instead of $\mathcal{P}(n)$, considered the set
$$\left\lbrace\left(\frac{ak^d}{n}, \begin{pmatrix}\sqrt{n^{-1}}&bk^d\sqrt{n^{-1}}\\&\sqrt{n}\end{pmatrix}\right)\in \T\times\Gamma\backslash G: (k,n)=1 \right\rbrace,$$
for some $a,b,d\in\N$ and let $n$ to run along any sequence with $(ab,n)=1$. One can generalize Theorem \ref{main-theorem} without much hassle to incorporate general $a,b,d$ like \cite{ELS} and let $n$ to run along any sequence with $(ab,n)=(d,\varphi(n))=1$. However, to remove the restriction $(d,\varphi(n))=1$ we need an estimate of a type of higher degree Ramanujan sum of the form $\sum_{(k,n)=1}e_l(k^d/n)$ (see proof of Lemma \ref{proof-cusp-form}) which we do not focus on in this paper.
\end{remark}

\begin{remark}
In the proof we are using Selberg's theorem that the non-constant automorphic forms for $\SL_2(\Z)$ are tempered at infinity. In general, if we only use (e.g. in \eqref{whittaker-bound}) that the representations are $\vartheta$-tempered at infinity (i.e. the real parts of the Langlands parameters lie in $[-\vartheta,\vartheta]$) then the error term in Theorem \ref{main-theorem} will be of the order of $n^{-1/2+\max(7/64,\vartheta)+\epsilon}$. In particular, this also shows that the error depends on the spectral gap as implicitly stated in \cite{ELS}.
\end{remark}

\begin{remark}
Recently, Burrin--Shapira--Yu \cite{BSY} proved a result similar to Theorem \ref{main-theorem} but only on the modular surface, with the same error rate 
The proof techniques of the analogous theorem in \cite{BSY} on the modular surface are essentially similar to ours, however, \cite{BSY} only has to consider the spherical forms.
\end{remark}

\section{A Few Preliminaries}
Before we dive into the proof of Theorem \ref{main-theorem} we describe below the preliminary tools which we need during the proof.

\subsection{Eisenstein Series}\label{sec:eisenstein-series}
We briefly recall the theory of Eisenstein series on $\Gamma\backslash G$. 

For $s\in \C$ and $\delta\in\{0,1\}$ we denote the principal series representation 
$$I_{s,\delta}:=\mathrm{Ind}_{B}^G|.|^s\sgn^\delta\boxplus|.|^{-s}\sgn^\delta,$$ 
where $B$ is the Borel subgroup of $G$ of upper triangular matrices. Let $h$ be any holomorphic section in the induced representation $I_{s,\delta}$. So, for any $g\in G$ we have
\begin{equation}\label{transformation-principal-series}
    h\left[n(x)\begin{pmatrix}y&\\&1\end{pmatrix}g\right]=\sgn(y)^\delta|y|^{s}h(g),\quad y\in\R^\times,\quad n(x):=\begin{pmatrix}1&x\\&1\end{pmatrix}\in N,
\end{equation}
where $N$ is the unipotent subgroup $G$ of upper triangular matrices. 
If $\Re(s)=1/2$ then $I_{s,\delta}$ is unitary and a $G$-invariant inner product can be given by
$$\langle h_1,h_2\rangle:=\int_{K}h_1(k)\overline{h_2(k)}dk,$$
where $K:=\mathrm{SO}_2(\R)$ is the maximal compact in $G$ equipped with the probability Haar measure $dk$.

We define the Eisenstein series attached to $h$ by
\begin{equation}\label{first-defn-eis}
    \Eis(h)(g):=\sum_{\gamma\in \Gamma\cap N\backslash \Gamma}h(\gamma g),
\end{equation}
which is absolutely convergent for $\Re(s)>1$. One can meromorphically continue $\Eis(h)$ for all $s\in\C$.
Let $\xi(s):=\Gamma_\R(s)\zeta(s)$, where $\Gamma_\R(s):=\pi^{-s/2}\Gamma(s/2)$, be the completed zeta function. One can write the Fourier expansion 
$$\Eis(h)(g):
=\Eis(h)_0(g)+\sum_{m\neq 0}\frac{\lambda_s(m)}{\sqrt{|m|}}\sgn(m)^\delta W_h\left[\begin{pmatrix}m&\\&1\end{pmatrix}g\right],$$
where the Fourier coefficients are given by
\begin{equation}\label{fourier-coeff-e-series}
\lambda_s(m):=\frac{|m|^{1/2-s}\sigma_{2s-1}(|m|)}{\zeta(2s)},\quad \sigma_z(n):=\sum_{d\mid n}d^z.
\end{equation}
$W_h$ is the Whittaker function attached to $h$ which, for $\Re(s)>1/2$, is given by
$$W_h(g):=\int_{\R}h(wn(x)g)e(-x)dx, \quad w:=\begin{pmatrix}&-1\\1&\end{pmatrix};$$
and its analytic continuation to the whole complex plane.
$\Eis(h)_0$ is the constant term of the Eisenstein series which is given by
\begin{equation}\label{zeroth-coeff}
    \Eis(h)_0(g):=h(g)+\frac{\xi(2s-1)}{\xi(2s)}M_sh(g),
\end{equation}
where $M_s$ is the normalized intertwiner from $I_{s,\delta}$ to $I_{1-s,\delta}$ which, for $\Re(s)>1/2$, is given by,
$$M_sh(g):=\frac{\Gamma_\R(2s)}{\Gamma_\R(2s-1)}\int_{\R}h(wn(x)g)dx.$$
Again, the above has meromorphic continuation to the whole complex plane. It is known that $\Eis(h)$ is holomorphic if $h\in I_{s,\delta}$ with $\Re(s)=1/2$.

The theory of Fourier expansion of the Eisenstein series $\Eis(h)$ is widely available in the literature when $h$ is a spherical vector. If $h$ is a non-spherical vector then the above form of Fourier expansion maybe obtained from \cite[\S4.1.7]{MV}. However, for readers' convenience here we provide a short sketch of a proof of the Fourier expansion when 

By the abelian Fourier theory on $\Z\backslash\R$ we have
$$\Eis(h)(g)=\sum_{m\in\Z}\widehat{\Eis}(h)_m(g),$$
where 
$$\widehat{\Eis}(h)_m(g):=\int_{\Z\backslash\R}\Eis(h)(n(x)g)e(-mx)dx=\int_{n(x)\in\Gamma\cap N\backslash N}\Eis(h)(n(x)g)e(-mx)dx.$$
We let $\Re(s)$ to be sufficiently large so that we can write $\Eis(h)$ in the form as in \eqref{first-defn-eis}. We use Bruhat decomposition of $\Gamma$ to write
\begin{multline*}
    \widehat{\Eis}(h)_m(g)=\int_{\Z\backslash\R}h(n(x)g)e(-mx)dx\\+\sum_{c=1}^\infty\sum_{(d,c)=1}\int_{\R}h\left[n(\bar{d}/c)\begin{pmatrix}c^{-1}&\\&c\end{pmatrix}wn(d/c)n(x)g\right]e(-mx)dx.
\end{multline*}
Here $d\bar{d}\equiv 1\mod c$.
We use \eqref{transformation-principal-series} to obtain that the first term vanishes unless $m=0$, in which case the quantity equals to $h(g)$. Doing a few change of variables and using \eqref{transformation-principal-series} a couple of times we obtain that the second term equals to
$$\sum_{c=1}^\infty\frac{1}{c^{2s}}\int_{\R}h\left(wn(x)g\right)e(-x)dx$$
for $m=0$, and
$$\sgn(m)^\delta|m|^{s-1}\sum_{c=1}^\infty\frac{\sum_{(d,c)=1}e(dm/c)}{c^{2s}}\int_{\R}h\left[wn(x)\begin{pmatrix}m&\\&1\end{pmatrix}g\right]e(-x)dx$$for $m\neq 0$.
We recall a classical result that for $\Re(s)$ sufficiently large
$$\sum_{(d,c)=1}\frac{e(dm/c)}{c^{2s}}=\begin{cases}
\frac{\sigma_{2s-1}(|m|)}{|m|^{2s-1}\zeta(2s)},&\text{if }m\neq 0\\
\frac{\zeta(2s-1)}{\zeta(2s)},&\text{if }m=0.
\end{cases}$$
We conclude the sketch by analytic continuation.

The automorphic representation $\pi_{s,\delta}$ generated by $\Eis(h)$ for $h\in I_{s,\delta}$ with $\Re(s)=1/2$ is unitary and we have that $\langle \Eis(h_1),\Eis(h_2)\rangle_{\pi_{s,\delta}}$ is proportional to $\langle h_1,h_2\rangle$
where the proportionality constant depend only on $s$ and $\delta$.

\subsection{Sobolev norms}\label{sec:sobolev-norm}
Let $\{X_1,X_2,X_3\}$ be a basis of $\Lie(G)$. We define a Laplacian on $G$ by
$$\D:=1-X_1^2-X_2^2-X_3^2$$
For an irreducible unitary (local or automorphic) representation $\pi$ of $G$ we define a Sobolev norm on $\pi$ by
$$S_{d}(v):=\|\D^dv\|_{\pi}, \quad v\in \pi, d\in \N.$$
It is known that $\D$ is self-adjoint and positive definite on the unitary representations of $G$ and there exists a $d>0$ such that $\D^{-d}$ is of trace class, see \cite{NS}. We also need the following result.

Recall the principal series $I_{1/2+s,\delta}$ for $\Re(s)=0$. Note that, if $h\in I_{1/2+s,\delta}$ is an eigenfunction of $\D$ with eigenvalue $\nu_h$ then $\Eis(h)$ is also an eigenfunction of $\D$ with eigenvalue $\nu_h$. This implies that for each $d_1$ there is a $d_2$ and $d_3$ such that
$$S_{d_2}(h)\ll S_{d_1}(\Eis(h))\ll S_{d_3}(h).$$
Here in the first and third quantities the Sobolev norms are on the local representation $I_{1/2+s,\delta}$ while in the middle quantity the same is on the automorphic representation generated by $\Eis(h)$. From now on we will not specify whether $S_d$ is considered on a local or an automorphic representation and hope that it will be clear from the context.

\begin{lemma}\label{trace-class-property}
For every $d_0$ there exists a $d>d_0$ such that
$$\int_{(0)}(1+|s|)^{d_0}\sum_{h\in \B(I_{1/2+s,\delta})}\nu_h^{-d}ds\ll 1,$$
where $\B(I_{1/2+s,\delta})$ denotes the orthonormal basis of $I_{1/2+s,\delta}$ consisting of eigenfunctions of $\D$ and $\nu_h$ are the eigenvalues of $h\in\B(I_{1/2+s,\delta})$ under $\D$.
\end{lemma}

\begin{proof}
We have $\D=1-C_G+2C_K$ where $C_H$ is the Casimir operator for the group $H$. This can be seen by choosing a basis $\{X,Y,W\}$ of $\Lie(G)$ such that $\{X,Y,W\}$ is orthogonal with respect to the Killing form and and $W$ is a basis of $\Lie(K)$. Then by the definition of the standard Cartan involution one has $C_G=X^2+Y^2-W^2$ and $C_K=-W^2$ and the claim follows.

So $h$, being an eigenvector of $\D$, is also an eigenvector of $C_K$. Thus we may choose $\B(I_{1/2+s,\delta})$ to be a $K$-type orthonormal basis. Then for $h$, a $k$-type for $k\in\Z$, we calculate that (see \cite[\S2.6]{Bum})
$$\nu_h \asymp 1+ |s|^2 +k^2.$$
Thus the integral in the statement is bounded by
$$\int_{(0)}(1+|s|)^{d_0}\sum_{k\in\Z}(1+|s|^2+k^2)^{-d}d|s|.$$
The above is convergent for large enough $d>0$.
\end{proof}

If $\Re(s)=0$ and $I_{1/2+s,\delta}\ni h$ is a principal series representation then for a fixed compact set $\Omega\subset G$ we have
\begin{equation}\label{sobolev-estimate}
\|h\|_{L^{\infty}(\Omega)}, \|M_sh\|_{L^{\infty}(\Omega)} \ll_\Omega S_d(h),
\end{equation}
for some fixed $d$. The first inequality follows by the classical Sobolev embedding and the second inequality follows from \cite[\S4.1.8]{MV}.

We also define a Sobolev norm on $C^\infty_c(\T)\otimes C^\infty_c(\Gamma\backslash G)$ by
$$S_{d_1,d_2}(f_1\otimes f_2):=S_{d_1}(f_1)S_{d_2}(f_2):=\|\partial^{d_1}f_1\|_2\|\D^{d_2}f_2\|_2,\quad d_1,d_2\in \Z_{\ge 0},$$
where $\partial$ is the differential operator $d/dx$ on $\R$ and $\|.\|_2$ denotes the $L^2$-norm on the corresponding space with respect to its invariant probability measure.

\subsection{Spectral decomposition}
We describe the pointwise spectral decomposition of $\Psi\in C_c^\infty(\Gamma\backslash G)\subset L^2(\Gamma\backslash G)$, see \cite[\S2.2.1]{MV}.
\begin{equation}\label{spectral-decomposition}
    \Psi(g)=\langle \Psi,1\rangle +\sum_{\pi\textrm{ cuspidal}}\sum_{\phi\in \B(\pi)}\frac{\langle \Psi,\phi\rangle\phi(g)}{\|\phi\|^2}+\int_{(0)}\sum_{\delta\in\{0,1\}}\sum_{h\in \B(I_{1/2+s,\delta})}\frac{\langle \Psi,\Eis(h)\rangle \Eis(h)(g)}{\|h\|^2}\frac{ds}{4\pi i},
\end{equation}
$\langle,\rangle$ is the usual $L^2$ inner product on $\Gamma\backslash G$ and $\B(\sigma)$ denotes an orthogonal basis of $\sigma$. We may choose $\B(\sigma)$ to contain eigenforms of the full Hecke algebra. The RHS of \eqref{spectral-decomposition} does not depend on the choices of bases and converges absolutely and uniformly on compacta.

\subsection{Whittaker functions}
Let $\pi$ be a cuspidal representation for $\SL_2(\Z)$ and $\phi\in\pi$ be a Hecke eigenform. Then $\phi$ has a Fourier expansion. We normalize $\phi$ so that the first Fourier coefficient of $\phi$ is one. We can write
\begin{equation}\label{fourier-expansion}
    \phi(g)=\sum_{m\neq 0}\frac{\lambda_\pi(m)}{\sqrt{|m|}}W_\phi\left[\begin{pmatrix}m&\\&1\end{pmatrix}g\right],
\end{equation}
where $W_\phi$ is the Whittaker function attached to $\phi$ defined by
$$W_\phi(g):=\int_0^1\phi(n(x)g)e(-x)dx,$$
and $\lambda_\pi(m)$ are the Hecke eigenvalues attached to $\pi$.

Let $W_v$ be the Whittaker function attached to $v$ which is either a cusp form $\phi$ or a unitary Eisenstein series $\Eis(h)$ attached to a principal series vector $h\in I_{s,\delta}$ with $\Re(s)=1/2$. Then $W_v$ is smooth if $v$ is a smooth vector and satisfies
\begin{equation}\label{unipotent-equivariance}
    W_v(n(x)g)=e(x)W_v(g),\quad n(x)\in N.
\end{equation}
$W_v$ also has the following decay property:
\begin{equation}\label{whittaker-bound}
    W_v\left[\begin{pmatrix}y&\\&1\end{pmatrix}\right]\ll_A \min(|y|^{1/2-\epsilon},|y|^{-A}) S_d(v),
\end{equation}
where $S_d(v)$ is the Sobolev norm defined in \S\ref{sec:sobolev-norm} and $d$ depends only on $A$. 

This result follows from \cite[Proposition $3.2.3$]{MV}. There the vector $W_v$ is considered to lie in a local generic representation $\sigma$ of $G$ and the corresponding Sobolev norm, which we name $S^\sigma$, is defined with respect to the Kirillov model of $\sigma$. In our case, when $v=\phi\in\pi$, a cusp form, we use that $\pi$ is isomorphic to its Whittaker model $\mathcal{W}(\pi)$ with respect to the character $N\ni n(x)\mapsto e(x)$ and 
$$\|\phi\|^2_\pi\asymp L(1,\mathrm{Ad},\pi)\|W_\phi\|^2_{\mathcal{W}(\pi)}.$$
Using the standard bounds of $L(1,\mathrm{Ad},\pi)$ (e.g. polynomials in the parameters of $\pi$) we conclude that for each $d_1$ there is a $d_2$ such that $S^\sigma_{d_1}(W_\phi)\ll S_{d_2}(\phi)$. Thus the bound in \eqref{whittaker-bound} follows from \cite[Proposition $3.2.3$]{MV}. Similarly, when $v=h$ a unitary principal series vector the same argument works using the isomorphism of the representation and its Kirillov model.

Also note that, we are using the fact, due to Selberg, that the non-constant automorphic forms for $\SL_2(\Z)$ are tempered at the archimedean place which implies that $\vartheta=0$ in \cite[Proposition $3.2.3$]{MV}.

\section{Cuspidal Spectrum} 
In this section we prove \eqref{effective-equidistribution} for $f_1$ being the Fourier mode $e_l$ for $l\in\Z_{\ge 0}$ and $f_2$ being a cuspidal Hecke eigenform $\phi$ lying in a cuspidal representaion $\pi\subseteq L^2(\Gamma\backslash G)$.

We arithmetically normalize $\phi$ i.e. normalize so that its first Fourier coefficient is one. Then the Fourier coefficients $\lambda_\pi(m)$ satisfy the following bound, see \cite{KS},
\begin{equation}\label{hecke-eigenvalue-bound}
    \lambda_\pi(m)\ll_\epsilon |m|^{7/64+\epsilon},
\end{equation}
where the implied constant is uniform in $\phi$. If $\pi$ is a discrete series then the above bound can be improved to
$$\lambda_\pi(m)\ll_\epsilon |m|^\epsilon,$$
which follows from Deligne's result.

\begin{lemma}\label{proof-cusp-form}
Let $l\in\Z_{\ge 0}$ and $\phi$ be an arithmetically normalized cuspidal Hecke eigenform on $\Gamma\backslash G$. Then
$$\frac{1}{\varphi(n)}\sum_{(k,n)=1}e_l\left(\frac{k}{n}\right)\phi\left[\begin{pmatrix}1&\frac{k}{n}\\&1\end{pmatrix}\begin{pmatrix}\sqrt{n^{-1}}&\\&\sqrt{n}\end{pmatrix}\right]\ll_\epsilon n^{-1/2+7/64+\epsilon}S_d(\phi),$$
for some fixed $d$.
\end{lemma}

\begin{proof}
Using \eqref{fourier-expansion} and \eqref{unipotent-equivariance} we obtain that the LHS of the estimate in the lemma is
\begin{equation}\label{lhs-for-cusp-form}
    \sum_{m\neq 0}\frac{\lambda_\pi(m)}{\sqrt{|m|}}W_\phi\left[\begin{pmatrix}m/n&\\&1\end{pmatrix}\right]r_n(m+l),
\end{equation}
where $r_n$ is the normalized Ramanujan sum defined by
$$r_n(m'):=\frac{1}{\varphi(n)}\sum_{(a,n)=1}e(am'/n).$$
It is a classical result that
\begin{equation*}
    r_n(m')=\frac{\mu(n/(n,m'))}{\varphi(n/(n,m'))},
\end{equation*}
where $\mu$ is the M\"obius function. We rearrange the absolutely convergent sum in \eqref{lhs-for-cusp-form} and estimate by
$$\le \sum_{d\mid n}\frac{1}{\varphi(n/d)}\sum_{m\equiv -l\mod d}\frac{|\lambda_\pi(m)|}{\sqrt{|m|}}\left|W_\phi\left[\begin{pmatrix}m/n&\\&1\end{pmatrix}\right]\right|.$$
We write $m=dk-l$, apply \eqref{hecke-eigenvalue-bound}, and \eqref{whittaker-bound} with $v=\phi$ and $A=1/2+7/64+2\epsilon$ to obtain that the above sum is 
\begin{multline*}
\ll \sum_{d\mid n}\frac{1}{\varphi(n/d)}\left[\sum_{k:|k-l/d|<1}|dk-l|^{-1/2+7/64+\epsilon}(|dk-l|/n)^{1/2-\epsilon}\right.\\+\left.\sum_{k:|k-l/d|\ge1}|dk-l|^{-1/2+7/64+\epsilon}(|dk-l|/n)^{-1/2-7/64-2\epsilon}\right].
\end{multline*}
Among the inner sums above the first one is $\ll n^{-1/2+7/64+\epsilon}$ and the second one is $\ll n^{1/2+7/64+2\epsilon}d^{-1-\epsilon}$. We use that $\varphi(n/d)\gg (n/d)^{1-\epsilon}$ to estimate that the above is
$$\ll_\epsilon n^{-1/2+7/64+\epsilon}\sum_{d\mid n}d^{-\epsilon},$$
hence the result follows.
\end{proof}

\section{Continuous Spectrum}
Now we focus on the continuous spectrum. We recall the notations developed in \S\ref{sec:eisenstein-series}.

\begin{lemma}\label{proof-e-series}
Let $l\in\Z_{\ge 0}$ and $h\in I_{1/2+s,\delta}$ with $\Re(s)=0$. Then 
$$\frac{1}{\varphi(n)}\sum_{(k,n)=1}e_l\left(\frac{k}{n}\right)\Eis(h)\left[\begin{pmatrix}1&\frac{k}{n}\\&1\end{pmatrix}\begin{pmatrix}\sqrt{n^{-1}}&\\&\sqrt{n}\end{pmatrix}\right]\ll_\epsilon n^{-1/2+\epsilon}S_d(h)(1+|s|)^\epsilon,$$
for some fixed $d$.
\end{lemma}

\begin{proof}
We proceed as in the proof of Lemma \ref{proof-cusp-form}. Using the Fourier expansion of $\Eis(h)(g)$ in \S\ref{sec:eisenstein-series}, description of the constant term \eqref{zeroth-coeff}, and \eqref{transformation-principal-series} we obtain that the LHS of the estimate in the lemma is
\begin{multline}\label{lhs-for-e-series}
    n^{-1/2}\left(n^{-s}h(1)+n^{s}\frac{\xi(2s)}{\xi(1+2s)}M_{1/2+s}h(1)\right)r_n(l)\\
    +\sum_{m\neq 0}\frac{\lambda_{1/2+s}(m)}{\sqrt{|m|}}\sgn(m)^\delta W_h\left[\begin{pmatrix}m/n&\\&1\end{pmatrix}\right]r_n(m+l).
\end{multline}

We start with the second summand in \eqref{lhs-for-e-series}. Recall the  For $\Re(s)=0$ we have the standard estimates
\begin{equation*}
    \sigma_{2s}(m)\ll m^\epsilon,\quad \zeta(1+2s)^{-1}\ll_\epsilon |s|^\epsilon,
\end{equation*}
which imply that
\begin{equation}\label{fourier-coeff-bound}
    \lambda_{1/2+s}(m)\ll_\epsilon {|ms|^\epsilon}.
\end{equation}
We proceed similarly as in the proof of Lemma \ref{proof-cusp-form} but use \eqref{fourier-coeff-bound} instead of \eqref{hecke-eigenvalue-bound}. We obtain that 
\begin{equation}\label{lhs-3}
    \sum_{m\neq 0}\frac{\lambda_{1/2+s}(m)}{\sqrt{|m|}}\sgn(m)^\delta W_h\left[\begin{pmatrix}m/n&\\&1\end{pmatrix}\right]r_n(m+l)\ll_{s,\epsilon} n^{-1/2+\epsilon}S_d(h)(1+|s|)^\epsilon,
\end{equation}
for some fixed $d$.

Now we focus on the first quantity of \eqref{lhs-for-e-series}. From the functional equation of the Riemann zeta function one gets that
$$\frac{\xi(2s)}{\xi(1+2s)}=\frac{\xi(1-2s)}{\xi(1+2s)}\asymp 1,$$
for $\Re(s)=0$. Using \eqref{sobolev-estimate} we obtain
$$\left(n^{-s}h(1)+n^{s}\frac{\xi(2s)}{\xi(1+2s)}M_{1/2+s}h(1)\right)r_n(l)\ll S_d(h),$$
for $\Re(s)=0$ and some fixed $d$. Hence we conclude.
\end{proof}

\section{Proof of Theorem \ref{main-theorem}}

Let $f_1\in C^\infty_c(\T)$ with the Fourier expansion
$$f_1(x)=\sum_{l\in \Z}\hat{f}_1(l)e_l(x)=\int_{\T} f_1+\sum_{l\neq 0}\hat{f}_1(l)e_l(x).$$
Then,
$$\frac{1}{\varphi(n)}\sum_{(k,n)=1}f_1(k/n)=\int_{\T} f_1+\sum_{l\neq 0}\hat{f}_1(l)r_n(l).$$
We have a rapid decay of the Fourier coefficients of $f_1$ i.e.
\begin{equation}\label{rapid-decay-fourier}
    \hat{f}_1(l)\ll_{A}(1+|l|)^{-A}S_{A}(f_1).
\end{equation}
We use that 
$$r_n(l)\ll \varphi(n/(n,l))^{-1}\ll_{\epsilon} (n/l)^{-1+\epsilon},$$
and \eqref{rapid-decay-fourier}
to obtain that
\begin{equation}\label{constant-term-torus}
    \frac{1}{\varphi(n)}\sum_{(k,n)=1}f_1(k/n)=\int_{\T} f_1+n^{-1+\epsilon}S_d(f_1),
\end{equation}
for some fixed $d$.

\begin{proof}[Proof of Theorem $1$]
We will prove \eqref{effective-equidistribution}.
Recall spectral decomposition \eqref{spectral-decomposition} and apply to $f_2$. Note that, $\langle f_2,1\rangle=\int_{\Gamma\backslash G}f_2$. 

If $\pi$ is cuspidal then we choose $\B(\pi)$ consisting of eigenfunctions of $\D$ which are also Hecke eigenfunctions. We integrate by parts with respect to $\D$ several times and then apply Cauchy--Schwarz to obtain
\begin{equation}\label{rapid-decay-cusp}
    \langle f_2,\phi\rangle \ll_{d, f_2} |\nu_\phi|^{-A}\|\phi\| S_A(f_2),
\end{equation}
for any $A>0$. Here $\nu_\phi$ is the eignevalue of $\phi$ under $\D$ and we are using that
$S_d(\phi)=\|\D^d\phi\|=|\nu_\phi|^d\|\phi\|$.

Similarly, for $\Re(s)=0$ we choose $\B(I_{1/2+s,\delta})$ consisting of an orthogonal eigenbasis of $\D$. Clearly, $\Eis(h)$ is an eigenfunction of $\D$ with the same eigenvalue as of $h$ under $\D$. We first integrate by parts with respect to $\D$ in the integral of $\langle f_2,\Eis(h)\rangle$. Then we use the Fourier expansion of $\Eis(h)$ as in \S\ref{sec:eisenstein-series}, use \eqref{sobolev-estimate}, and work as in the proof of Lemma \ref{proof-e-series} to obtain
$$\Eis(h)\mid_{\supp(f_2)}\ll_{\supp(f_2)} S_d(h)(1+|s|)^\epsilon.$$
Hence we deduce that
\begin{equation}\label{rapid-decay-eis}
    \langle f_2,\Eis(h)\rangle \ll_{\epsilon,\supp(f_2)} |\nu_h|^{-A}S_{A'}(f_2)\|h\| (1+|s|)^\epsilon,
\end{equation}
for any $A>0$ and some $A'$ depending only on $A$. Here $\nu_h$ is the eigenvalue of $h$ under $\D$.

Finally, using \eqref{spectral-decomposition}, Lemma \ref{proof-cusp-form}, Lemma \ref{proof-e-series}, and \eqref{constant-term-torus} we obtain that
\begin{multline*}
    \frac{1}{\varphi(n)}\sum_{(k,n)=1}f_1\left(\frac{k}{n}\right)f_2\left[\begin{pmatrix}1&\frac{k}{n}\\&1\end{pmatrix}\begin{pmatrix}\sqrt{n^{-1}}&\\&\sqrt{n}\end{pmatrix}\right]-\int_\T f_1\int_{\Gamma\backslash G}f_2\\
    \ll_{\epsilon,\supp(f_2)} n^{-1+\epsilon} S_d(f_1)\|f_2\|_{2}+n^{-1/2+7/64+\epsilon}\sum_{l\in\Z}\hat{f}_1(l)\sum_{\pi\textrm{ cuspidal}}\sum_{\phi\in \B(\pi)}\frac{|\langle f_2,\phi\rangle|}{\|\phi\|}|\nu_\phi|^d\\
    +n^{-1/2+\epsilon}\sum_{l\in\Z}\hat{f}_1(l)\int_{(0)}(1+|s|)^{\epsilon}\sum_{\delta\in\{0,1\}}\sum_{f\in \B(I_{1/2+s,\delta})}\frac{|\langle f_2,\Eis(h)\rangle|}{\|h\|}|\nu_h|^d d|s|,
\end{multline*}
for some fixed $d$. We apply \eqref{rapid-decay-cusp} and \eqref{rapid-decay-eis} to the automorphic spectral integrals and apply \eqref{rapid-decay-fourier} to the $l$-sum with large enough $A$. We confirm that the above sums and integrals are absolutely convergent by the trace class property of $\D^{-A}$ and Lemma \ref{trace-class-property}. We conclude the proof by noting the dependencies on the Sobolev norms of $f_1\otimes f_2$ from \eqref{rapid-decay-cusp}, \eqref{rapid-decay-eis}, and \eqref{rapid-decay-fourier}.
\end{proof}

\begin{ackno}
We thank Manfred Einsiedler, Manuel Luethi, and Paul Nelson for several helpful conversations and feedback on an earlier draft. We acknowledge the anonymous referee for carefully reading the manuscript and suggesting several improvements. We also thank ETH Z\"urich where the work was mostly done while the author was a doctoral student there.
\end{ackno}


\begin{thebibliography}{99}

\bibitem{Bum}
D. Bump, \emph{Automorphic forms and representations}, Cambridge Studies in Advanced Mathematics, 55. Cambridge University Press, Cambridge, 1997. xiv+574 pp. ISBN: 0-521-55098-X.

\bibitem{BSY}
Burrin, C.; Shapira, U.; Yu, S.: \emph{Translates of rational points along expanding closed horocycles on the modular surface}, arXiv:2009.13608 [math.DS].

\bibitem{ELS}
Einsiedler, M.; Luethi, M.; Shah, N.: \emph{Primitive rational points on expanding horocycles in products of the modular surface with the torus}, to appear in Ergodic Theory and Dynamical Systems, arXiv:1901.03078 [math.DS].


\bibitem{KS}
Kim, H. H.; Sarnak, P.: \emph{Refined estimates towards the Ramanujan and Selberg conjectures,
appendix to H. H. Kim, Functoriality for the exterior square of $\GL_4$ and the symmetric fourth
of $\GL_2$, with appendix 1 by Dinakar Ramakrishnan and appendix 2 by Kim and Peter Sarnak},
J. Amer. Math. Soc., 16(1), 2003. 1

\bibitem{MV}
Michel, P.; Venkatesh, A.: \emph{The subconvexity problem for $\GL_2$}, Publ. Math. Inst. Hautes Études Sci. No. 111 (2010), 171-271. 

\bibitem{NS}
Nelson, E.; Stinespring, W. F.: \emph{Representation of elliptic operators in an enveloping algebra}, Amer. J. Math. 81 1959 547–560. 

\end{thebibliography}
\end{document}